\documentclass{amsart}
\usepackage{amssymb,amsmath,amscd, amsfonts,latexsym}
\usepackage{graphicx}
\usepackage{amssymb}
\usepackage{epstopdf}

\usepackage{verbatim}  

\usepackage[all]{xy}
\usepackage{color}
\newtheorem{theorem}{Theorem}

\newtheorem{lemma}{Lemma}

\theoremstyle{remark}

\newcommand{\bbQ}{\mathbb Q}
\newcommand{\bbR}{\mathbb R}
\newcommand{\bbZ}{\mathbb Z}
\newcommand{\bbN}{\mathbb N}

\newcommand{\bbP}{\mathbb P}

\renewcommand{\geq}{\geqslant}
\renewcommand{\ge}{\geqslant}

\title[\title{On a variant of Pillai's problem II}]{On a variant of Pillai's problem II}

\author[K. C. Chim]{Kwok Chi Chim}

\author[I. Pink]{Istv\'an Pink}

\author[V. Ziegler]{Volker Ziegler}

\thanks{The first author was supported by the Austrian Science Fund (FWF) under the projects P26114 and W1230.
The second and the third authors were supported by the Austrian Science Fund (FWF) under the project P 24801-N26. }



\address{K. C. Chim \newline
         \indent Institute of Analysis and Number Theory, Graz University of Technology \newline
         \indent Kopernikusgasse 24/II \newline
         \indent A-8010 Graz, Austria}

\email{chim\char'100math.tugraz.at}

\address{I. Pink \newline
         \indent Institute of Mathematics, University of Debrecen \newline
         \indent H-4010 Debrecen, P.O. Box 12, Hungary}

\email{pinki\char'100science.unideb.hu; istvan.pink\char'100sbg.ac.at}

\address{V. Ziegler \newline
         \indent University of Salzburg \newline
         \indent Hellbrunnerstrasse 34/I \newline
         \indent A-5020 Salzburg, Austria}

\email{volker.ziegler\char'100sbg.ac.at}

\begin{document}

\begin{abstract}
In this paper, we show that there are only finitely many $c$ such that the equation $U_n - V_m = c$ has at least two distinct solutions $(n,m)$,
where $\{U_n\}_{n\geq 0}$ and $\{V_m\}_{m\geq 0}$ are given linear recurrence sequences. 
\end{abstract}

\maketitle

\section{Introduction}

A linear recurrence sequence is a sequence $\{U_n\}_{n\ge 0}$ such that for some $k\ge 1$, we have 
$$
U_{n+k}=c_1 U_{n+k-1}+\cdots+c_k U_n
$$
for all $n\ge 0$, where $c_1,\ldots,c_k$ are given complex numbers with $c_k\ne 0$. When $c_1,\ldots,c_k$ are integers and $U_0,\ldots,U_{k-1}$ are also 
integers, $U_n$ is an integer for all $n\ge 0$ and we say that $\{U_n\}_{n\ge 0}$ is defined over the integers. In what follows we will always assume that
$\{U_n\}_{n\ge 0}$ is defined over the integers.

It is known that if we write
$$
F(X)=X^k-c_1X^{k-1}-\cdots -c_k=\prod_{i=1}^t (X-\alpha_i)^{\sigma_i},
$$
where $\alpha_1,\ldots,\alpha_t$ are distinct complex numbers, and $\sigma_1,\ldots,\sigma_t$ are positive integers whose sum is $k$, then there exist 
polynomials $a_1(X),\ldots, a_t(X)$ whose coefficients are in $\bbQ(\alpha_1,\ldots,\alpha_t)$ such that $a_i(X)$ is of degree 
at most $\sigma_i-1$ for $i=1,\ldots,t$, and such that furthermore the formula
$$
U_n=\sum_{i=1}^t a_i(n)\alpha_i^n
$$
holds for all $n\ge 0$. We may certainly assume that $a_i(X)$ is not the zero polynomial for any $i=1,\ldots,t$.
We call $\alpha=\alpha_1$ a dominant root of $\{U_n\}_{n\ge 0}$, if $|\alpha_1|>|\alpha_2|\geq \dots \geq |\alpha_t|$.
In this case the sequence $\{U_n\}_{n\ge 0}$ is said to satisfy the dominant root condition.

This paper is a follow-up to our previous work \cite{SalzburgI:2016}, in which we found all integers $c$
admitting at least two distinct representations of the form $F_n - T_m$ for some positive integers $n\geqslant 2$ and $m\geqslant 2$.
Here we denote by $\lbrace F_n \rbrace_{n \geqslant 0}$ the sequence of Fibonacci numbers given by $F_0 = 0$, $F_1 = 1$ and 
$F_{n+2} = F_{n+1} + F_n$ for all $n \geqslant 0$, and denote by $\lbrace T_m \rbrace_{m\geqslant 0}$ the sequence of Tribonacci
numbers given by $T_0 = 0$, $T_1=T_2=1$ and $T_{m+3}=T_m+T_{m+1} + T_{m+2}$ for all $m \geqslant 0$. In \cite{SalzburgI:2016} 
the main result is the following:

\begin{theorem}\label{th:Fib-Tri}
The only integers $c$ having at least two representations of the form $F_n - T_m$ come from the set
$$\mathcal C=\{ 0, 1, -1, -2,  -3, 4, -5, 6, 8, -10, 11, -11, -22, -23, -41, -60, -271 \}.$$

Furthermore, for each $c\in\mathcal C$ all representations of the form $c=F_n - T_m$ with integers $n \geqslant 2$ and $m \geqslant 2$ are obtained. 
\end{theorem}

The above problem of obtaining all integers $c$ having at least 
two representations of the form $F_n - T_m$ can be regarded as a variant of Pillai's problem. 
Readers can refer to \cite{SalzburgI:2016} for the complete list of representations and some historical development of the Pillai's problem. 
The interested reader may also refer to the paper of Pillai~\cite{Pillai:1936} for the original problem, the papers of
Stroeker and Tijdeman \cite{Stroeker:1982} and Bennett \cite{Bennett:2001} for tackling special cases and the papers of Ddamulira, Luca and Rakotomalala \cite{Luca15}
and Bravo, Luca and Yaz\'{a}n \cite{Luca16} for other variants.

The purpose of this paper is to generalize Theorem \ref{th:Fib-Tri}. Assume that we are given two linear recurrence sequences $\{U_n\}_{n\ge 0}$ and
 $\{V_m\}_{m\ge 0}$ defined over the integers which satisfy the dominant root condition, then under some mild restrictions there exist only finitely many integers
 $c$ such that the equation
 $$ U_n-V_m=c$$
 has at least two distinct solutions $(n,m)\in\bbN\times\bbN$, where $\bbN=\{0,1,\dots,\}$ is the set of natural numbers. That is, we want to solve
\begin{equation}\label{eqt}
U_n-U_{n_1} = V_m - V_{m_1}
\end{equation}
for $(n,m) \neq (n_1,m_1)$. 

In order to avoid linear recurrence sequences such as $\lbrace 3, -3, 3, -3, \dotso \rbrace $ which would yield infinitely many solutions trivially,
we assume that both $\lbrace U_n \rbrace_{n \geqslant 0}$ and $\lbrace V_m \rbrace_{m \geqslant 0}$ are eventually strictly increasing in absolute values.
That is, we assume that there exist constants $N_0$ and $M_0$ such that $|U_{n+1}|> |U_n| > 0$ for all $n \geqslant N_0$ and
$|V_{m+1}| > |V_m| > 0$ for all $m \geqslant M_0$. We shall therefore require $n \geqslant N_0$ and $m \geqslant M_0$ when solving equation~\eqref{eqt}.

Throughout this paper, we denote by $C_0, C_1, \dotso ,C_{45}$ effectively computable constants. We prove the following theorem: 

\begin{theorem}\label{Main}
Suppose that $\lbrace U_n \rbrace_{n \geqslant 0} $ and $\lbrace V_m \rbrace_{m \geqslant 0}$ are two linear recurrence sequences defined over the integers
with dominant roots $\alpha$ and $\beta$ respectively. Furthermore, suppose that $\alpha$ and $\beta$ are multiplicatively independent. 
Suppose also that $\lbrace U_n \rbrace_{n \geqslant 0}$ and $\lbrace V_m \rbrace_{m \geqslant 0}$ are strictly increasing in absolute values
for $n \geqslant N_0$ and $m \geqslant M_0$ respectively.
Then there exists a finite set $\mathcal C$ such that the integer $c$ has at least two distinct representations of the form 
$U_n - V_m$ with $n \geqslant N_0$ and $m \geqslant M_0$, if and only if $c\in\mathcal C$. The set $\mathcal C$ is effectively computable.
\end{theorem}


Besides, the assumption that $\alpha$ and $\beta$ are multiplicatively independent is needed to avoid scenarios such as having 
$\lbrace U_n \rbrace_{n \geqslant 0} = \lbrace F_n \rbrace_{n \geqslant 0} $, $\lbrace V_m \rbrace_{m \geqslant 0} = \lbrace F_m \rbrace_{m \geqslant 0}$. 
In this case equation $c=F_{n+2} - F_{n+1} = F_{n+1} - F_{n-1}$ holds for all $n \geqslant 1$ and we have infinitely many $c$ that yield at least two solutions
to equation $ U_n-V_m=c$.

It should be also noted that the assumption that $\alpha$ and $\beta$ are multiplicatively independent is not necessary for the existence of only finitely many $c$. Consider the case
where $\lbrace U_n \rbrace_{n \geqslant 0} =\lbrace 2^n + 1 \rbrace_{n \geqslant 0}$ and $\lbrace V_m \rbrace_{n \geqslant 0} =\lbrace 4^m + 2\rbrace_{m \geqslant 0}$.
By elementary divisbility criteria one can easily verify that the only solutions to \eqref{eqt} with $n\neq n_1$ satisfy
$n=2m$ and $n_1=2m_1$, i.e. $c=-1$. Although \eqref{eqt} has infinitely many solutions the only $c$ such that $U_n-V_m=c$ has at least two solutions is $c=-1$.

In view of the two examples above it seems to be an intersting problem to relax the condition that $\alpha$ and $\beta$ are multiplicatively independent in Theorem \ref{Main}.

We shall prove Theorem \ref{Main} by applying the results of linear forms in logarithms and some results on the heights of algebraic numbers several
times to obtain an effectively computable upper bound for the value of the largest unknown among $\left\{n, m, n_1, m_1 \right\}$. 


\section{Preliminaries} \label{prelim}

In this section we present two basic tools needed in the proof of Theorem \ref{Main}. Firstly, we state a result on lower bounds of
linear forms in logarithms due to Baker and W\"ustholz~\cite{bawu93}. Secondly we provide a lower bound for the height of numbers of the
form $\frac{\alpha^n}{\beta^m}$ provided that $\alpha$ and $\beta$ are multiplicatively independent, and an upper bound for the height of $\frac{p(n)}{q(m)}$, where $p, q$ are arbitrary but fixed polynomials.

\subsection{A lower bound for linear forms in logarithms of algebraic numbers}

In 1993, Baker and W\"ustholz \cite{bawu93} obtained an explicit bound for linear forms in logarithms with a 
linear dependence on $\log B$, where $B \geqslant e$ denotes an upper bound for the height of the
linear form (to be defined later in this section).
It is a vast improvement compared with lower bounds with a dependence on higher powers of $\log B$ in preceding
publications by other mathematicians in particular Baker's original results \cite{Baker:1966}. The final structure
for the lower bound for linear forms in logarithms without an
explicit determination of the constant involved has been established by W\"ustholz \cite{wu88} and the precise determination of that constant 
(which is denoted as $C(n,d)$ in \cite{bawu93} and later in this section as $C(k,d)$) is the central aspect of \cite{bawu93}
(see also \cite{bawu07}). The improvement was mainly due to the use of the
analytic subgroup theorem established by W\"ustholz \cite{wu89}. We shall now state the result of Baker and W\"ustholz.

Denote by $\alpha_1, \dots, \alpha_k$ algebraic numbers, not $0$ or $1$, and by $\log \alpha_1, \dots, \log \alpha_k$
a fixed determination of their logarithms. Let $K=\bbQ(\alpha_1, \dotso, \alpha_k)$ and let $d=[K:\bbQ]$ be the degree of $K$ over $\bbQ$.
For any $\alpha \in K$, suppose that its minimal polynomial over the integers is
\[ g(x) = a_0 x^{\delta} + a_1 x^{\delta - 1} + \cdots + a_{\delta} = a_0 \prod_{j=1}^{\delta} (x - \alpha^{(j)})\]
where $\alpha^{(j)},\; j=1, \dotso, \delta$ are all the roots of $g(x)$. The absolute logarithmic Weil height of $\alpha$ is defined as
\[ h_0(\alpha) = \frac{1}{\delta} \left( \log |a_0| + \sum_{j=1}^{\delta} \log \left( \max \lbrace|\alpha^{(j)}|, 1 \rbrace \right)\right). \]
Then the modified height $h'(\alpha)$ is defined by
\[ h'(\alpha)=\frac{1}{d}\max \{h(\alpha), |\log\alpha|, 1\},\]
where $h(\alpha) = d h_0(\alpha)$ is the standard logarithmic Weil height of $\alpha$.

Let us consider the linear form
\[ L(z_1, \dotso, z_k)=b_1 z_1 + \cdots + b_k z_k, \]
where $b_1, \dotso, b_k$ are rational integers, not all 0 and define
\[ h'(L) = \frac{1}{d}\max\{h(L),1\}, \]
where $h(L) = d \log \left(\max_{1 \le j \le k} \left\{\frac{|b_j|}{b}\right\}\right)$ is the logarithmic Weil height of $L$,
with $b$ as the greatest common divisor of $b_1, \dotso, b_k$.
If we write $B=\max\lbrace|b_1|, \dotso, |b_k|, e \rbrace$, then we get
\[ h'(L)\leqslant \log B. \]

With these notations we are able to state the following result due to Baker and W\"ust\-holz~\cite{bawu93}.

\begin{theorem}\label{BaWu}
If ${\it\Lambda}=L(\log \alpha_1, \dotso, \log\alpha_k) \neq 0$, then
\[\log|{\it\Lambda}|\geqslant -C(k,d)h'(\alpha_1)\cdots h'(\alpha_k)h'(L), \]
where
\[ C(k,d)=18(k+1)!\,k^{k+1}(32d)^{k+2}\log(2kd). \]
\end{theorem}

With $|{\it\Lambda}| \leqslant \frac{1}{2}$, we have $\frac{1}{2}|{\it\Lambda}| \leqslant |{\it\Phi}| \leqslant 2|\it{\Lambda}|$, where
\[ {\it\Phi} = e^{{\it\Lambda}}-1 = \alpha_1^{b_1} \cdots \alpha_k^{b_k}-1, \]
so that
\begin{equation} \label{linearform}
\log|\alpha_1^{b_1} \cdots \alpha_k^{b_k}-1| \geqslant \log|{\it\Lambda}| - \log2.
\end{equation}

\subsection{Some results on heights}

Before we state our results let us recall some well known properties of the absolute logarithmic height:
\begin{equation*}
\begin{split}
h_0(\eta \pm \gamma) \leqslant & \ h_0(\eta)+h_0(\gamma)+\log{2},  \\
h_0(\eta\gamma^{\pm 1}) \leqslant & \ h_0(\eta)+h_0(\gamma),  \\
h_0(\eta^\ell)=& \ |\ell|h_0(\eta), \qquad \textrm{for} \; \ell \in \bbZ,
\end{split}
\end{equation*}
where $\eta,\gamma$ are some algebraic numbers. 

Upon applying inequality \eqref{linearform} from Theorem \ref{BaWu}, which is only valid for ${\it\Lambda}\neq 0$, we need to treat
the situation ${\it\Lambda}=0$ separately. We shall make use of the following lemma repeatedly applied when dealing with this situation.

\begin{lemma} \label{height}
Let $K$ be a number field and suppose that $\alpha,\beta\in K$ are two algebraic numbers which are multiplicatively independent. Moreover, let $n, m \in \bbZ$.
Then there exists an effectively computable constant $C_0 > 0$ such that
\[ h_0\left(\frac{\alpha^n}{\beta^m}\right) \geqslant C_0 \max \{|n|, |m|\}. \]
\end{lemma}

Although Lemma \ref{height} seems to be well known we found no apropriate reference. In order to keep the paper as self contained as possible
we give a proof of this Lemma.

Before we start with the proof of Lemma \ref{height} we want to fix some notations. Let $K$ be a number field. We denote by
$M_K$ the set of places of $K$. For each $v\in M_K$ we denote by $\|\cdot\|_v$ the normalized absolute value corresponding to $v$,
i.e., if $v$ lies above $p \in M_\bbQ := \{\infty\} \cup \bbP$, where $\bbP$ is the set of rational primes, then the restriction of
$\|\cdot\|_v$ to $\bbQ$ is $| \cdot|^{[K_v:\bbQ_p]/[K:\bbQ]}_p$, where $\bbQ_p$ and $K_v$ are the $p$-adic and $v$-adic completions of $\bbQ$ and $K$
respectively. Here, $|\cdot|_\infty$ is the usual absolute value and for a prime $p$ the norm $|\cdot|_p$ is the usual $p$-adic norm such that
$|p|_p=\frac 1p$.

Let us note that with these notations the product formula (see e.g. \cite[Chapter III, Theorem 1.3]{Neukirch:ANTE}) states that
$$\sum_{v\in M_K} \log \|\alpha\|_v =0$$
and the height can be written as
$$
h_0(\alpha)=\sum_{v\in M_K} \max\{0,\log \|\alpha\|_v\}.
$$

With these notations at hand we can turn  to the proof of Lemma \ref{height}.

\begin{proof}[Proof of Lemma \ref{height}]
Denote by $S\subseteq M_K$ the finite set of places where the valuation of either $\alpha$ or $\beta$ is non-zero. i.e.
\[ S= \{v \in M_K: \|\alpha\|_v \neq 0 \mbox{ or } \|\beta\|_v \neq 0 \}. \]
We consider a $\mbox{Log}$ function defined as follows:
$$
\mbox{Log}:K \longrightarrow \prod_{v \in S}{\bbR}\qquad \alpha \longmapsto \left(\log \|\alpha\|_v \right)_{v \in S}.
$$
Obviously,  $\mbox{Log}$ has the properties that 
$$
 \alpha^n \longmapsto n \: \mbox{Log} (\alpha),\quad \mbox{and}\quad
 \alpha \cdot \beta \longmapsto \mbox{Log}(\alpha) + \mbox{Log}(\beta),
$$
so that
\[ \mbox{Log} \left(\frac{\alpha^n}{\beta^m}\right) = n \: \mbox{Log}(\alpha) - m \:\mbox{Log}(\beta).\]
Since $\alpha$ and $\beta$ are multiplicatively independent, there exist valuations $v_1,v_2\in S$ such that the matrix 
\begin{equation*}
M= \left(
\begin{array}{cc}
\log \|\alpha\|_{v_1} & \log \|\beta\|_{v_1} \\
\log \|\alpha\|_{v_2} & \log \|\beta\|_{v_2} 
\end{array}\right)
\end{equation*}
is non-singular. For the moment let us write $A=\frac{\alpha^n}{\beta^m}$. If we consider the system of linear equations
\begin{equation*}
\begin{split}
n\log \|\alpha\|_{v_1} - m \log \|\beta\|_{v_1} & = \log \|A \|_{v_1}\\ 
n\log \|\alpha\|_{v_2} - m \log \|\beta\|_{v_2} & = \log \|A \|_{v_2},
\end{split}
\end{equation*}
we obtain from Cramer's rule that
\begin{align*}
|n| & \leqslant \frac{2 \max \{|\log \|A\|_{v_1}|,   |\log \|A\|_{v_2}| \} \cdot \max \{|\log \|\beta\|_{v_1}|,   |\log \|\beta\|_{v_2}| \} }{\det M},\\
|m| & \leqslant \frac{2 \max \{|\log \|A\|_{v_1}|,   |\log \|A\|_{v_2}| \} \cdot \max \{|\log \|\alpha\|_{v_1}|,   |\log \|\alpha\|_{v_2}| \} }{\det M}.
\end{align*}
From the above inequality, we have
$$
\max \left\{|\log \|A\|_{v_1}|, |\log \|A\|_{v_2}| \right\} \geqslant \max \left\{ \widetilde{C_1}|n| , \widetilde{C_2} |m|  \right\},
$$
where 
$$\widetilde{C_1} = \frac{\det M}{2  \max \{|\log \|\beta\|_{v_1}|,   |\log \|\beta\|_{v_2}| \} }>0$$
and
$$\widetilde{C_2} = \frac{\det M}{2  \max \{|\log \|\alpha\|_{v_1}|,   |\log \|\alpha\|_{v_2}| \} }>0.$$

As noted above we have that
\[ h_0(A) = \sum_{v \in M_K} \max \{ \log \|A\|_v, 0 \} \quad \mbox{and} \quad  \sum_{v \in M_K} \log \|A\|_v = 0.\]
From the product formula we deduce that there exists $v\in M_K$ such that 
\[ \log \|A\|_v \geqslant \frac{1}{|S|} \cdot \max\{ |\log \|A\|_{v_1}|,  |\log \|A\|_{v_2}| \} . \]
Thus, we obtain
\begin{align*}
h_0(A) = h_0\left(\frac{\alpha^n}{\beta^m}\right) & \geqslant \frac{1}{|S| }\max \left\{ |\log \|A\|_{v_1}|,  |\log \|A\|_{v_2}| \right\}\\
& \geqslant \frac{1}{|S| } \max \left\{\widetilde{C_1}|n| , \widetilde{C_2} |m| \right\}\\
& \geqslant C_0 \max \left\{|n| , |m| \right\},
\end{align*}  
where we may choose $C_0 = \frac{1}{|S| } \min\left\{\widetilde{C_1}, \widetilde{C_2} \right\}$.
\end{proof}

Let us also state the following result as a lemma:

\begin{lemma}\label{lem:po-height}
 Let $K$ be a number field and $p,q\in K[X]$ arbitrary but fixed polynomials.
 Then there exists an effectively computable constant $C=C(p,q)$ such that
 $$h_0\left( \frac{p(n)}{q(m)} \right)\leqslant C \log \max\{ n, m\}.$$
\end{lemma}

\begin{proof}
 Since $h_0\left( \frac{p(n)}{q(m)} \right)\leqslant h_0(p(n))+h_0(q(m))$ it suffices to prove that there exists an effectively computable constant $C$ such
 that $h_0(p(n))\leqslant C \log n$ for some fixed polynomial $p\in K[X]$. Assume that $p(n)=c_kn^k+\dots+c_1n+c_0$, then we have
 \begin{align*}
  h_0(p(n))&=h_0(c_kn^k+\dots+c_1n+c_0) \\
  &\leqslant h_0(c_k)+kh_0(n)+\dots+h_0(c_0)+k\log 2 \\
  &\leqslant C \log n. 
 \end{align*}
\end{proof}


\section{Proof of Theorem \ref{Main}} \label{proof}

\subsection{Set up}

Recall that we wish to solve equation \eqref{eqt}:
\begin{equation*} 
U_n -U_{n_1}  = V_m - V_{m_1},
\end{equation*}
for $(n,m) \neq (n_1,m_1)$, with $n, n_1 \geqslant N_0$ and $m, m_1 \geqslant M_0$.

We may assume that $m \neq m_1$, since otherwise $(n,m) = (n_1,m_1)$. Without loss of generality we may assume that $m>m_1$. But, then we have to distinguish
between the two cases $n>n_1$ and $n<n_1$. Since the proof of the second case is obtained by interchanging the roles of $n$ and $n_1$, i.e. to interchange $n_1$ and $n$
everywhere, we only give the proof of the first case. Therefore we assume from now on that
$n > n_1 \geqslant N_0$ and $m > m_1 \geqslant M_0$.

 In the following we use the $L$-notation. Assume $f(x)$, $g(x)$ and $k(x)$ are real functions and that $k(x) > 0$ for $x > 1$. We shall write
\[ f(x) = g(x) + L(k(x)) \] 
for
\[ g(x) - k(x) \leqslant f(x) \leqslant g(x) + k(x).  \]
The use of the $L$-notation is like the use of the $O$-notation but with the advantage to have an explicit bound for the error term. 

Let us consider the linear recurrence sequences $\{U_n\}_{n\geq 0}$ and $\{V_m\}_{m\geq 0}$ a bit closer.
Let us assume that the characteristic polynomials of $\{U_n\}_{n\geq 0}$ and $\{V_m\}_{m\geq 0}$
are
$$
 F_U(X)=\prod_{i=1}^t (X-\alpha_i)^{\sigma_i}\quad \mbox{and}\quad F_V(X)=\prod_{i=1}^s (X-\beta_i)^{\tau_i}
$$
respectively.

Let $\alpha$ and $\beta$ be the dominant roots of $\lbrace U_n \rbrace_{n\ge 0}$ and $\lbrace V_m \rbrace_{m\ge 0}$ respectively. According to our assumptions we can write
\begin{equation}\label{eq:asymp1}
\begin{split}
U_n&=a(n)\alpha^{n}+a_2(n)\alpha_2^n+\dots+a_t(n)\alpha_t^n\\
&=a(n)\alpha^{n}+L\left(a'' n^{A}\alpha_2^n\right)\\
&=a(n)\alpha^{n}+L(a' {\alpha'}^n)
\end{split}
\end{equation}
where $a',a'',A$ are suitable but effectively computable, non-negative constants, $a(X)$, $a_i(X) \in \bbQ(\alpha_1,\dots,\alpha_t)[X]$, $2 \leqslant i \leqslant t$
and $\alpha'\in \bbR$ is such that $|\alpha_1|=|\alpha|>\alpha'>|\alpha_2|$. Note that in case that $t=1$ we put $\alpha_2=1$ and $a'=a''=A=0$ and with this choice \eqref{eq:asymp1} still holds.
Let us also note that by our assumption that $\{U_n\}_{n\ge 0}$ is non-degenerate and defined over the integers the dominant root $\alpha$ is a real algebraic integer which is not a root of unity,
hence we have $|\alpha|>1$. Thus we may assume that also $|\alpha|>\alpha'>1$ holds. This also implies that $\{|U_n|\}_{n\ge 0}$ is eventually strictly increasing.
Moreover we may assume that $|a(n)|$ is increasing for all $n\geq N_1$ for some suitable
constant $N_1$. In addition, we choose $N_1$ large enough such that $|a(n)|\geqslant |a(n')|$ for all $n>N_1$ and $n>n'>0$.

Similarly we may write
\begin{equation}\label{eq:asymp2}
V_m=b(m) \beta^{m}+L(b' {\beta'}^n)
\end{equation}
where $b',\beta'$ are suitable constants. By the same arguments as above we may also assume that $|\beta|>\beta'>1$ and $|b(m)|$ is increasing provided that $m\geqslant M_1$, where $M_1$ is some
sufficiently large number. Moreover we assume that $M_1$ is chosen large enough such that $|b(m)|\geq |b(m')|$ for all $m\geqslant M_1$ and $m>m'>0$.

Without loss of generality, let us assume that $|\alpha| > |\beta|$. We denote by $\sigma$ and $\tau$ the degree of $a(n)$ and $b(m)$ respectively. Besides, we know that $|U_n|\sim a n^\sigma |\alpha^n|$ as
$n\rightarrow \infty$, where $a$ is the leading coefficient of $a(n)$. Similarly we know that $|V_m|\sim b m^\tau |\beta^m|$ as $m\rightarrow \infty$, where $b$ is the leading coefficient of $b(n)$. Therefore
there are positive constants $C_1,C_2$ and $C_3,C_4$ such that $C_2/C_1<|\alpha|$ and $C_4/C_3<|\beta|$ with
\begin{align*}
C_1 n^\sigma |\alpha|^n & \leqslant |U_n| \leqslant C_2 n^\sigma |\alpha|^n  && \mbox{for all } n \geqslant N_2 \\
C_3 m^\tau|\beta|^m & \leqslant |V_m| \leqslant C_4 m^\tau |\beta|^m &&  \mbox{for all } m \geqslant M_2,
\end{align*}
where $N_2$ and $M_2$ are sufficiently large.

Let us assume for the moment that $n>n_1\geq N_2$ and $m>m_1\geq M_2$. Using equation \eqref{eqt} we get that
\begin{equation*} 
\begin{split}
| U_n-U_{n_1}| & \leqslant |U_n| + |U_{n_1}| \leqslant C_2 n^\sigma\left(|\alpha|^{n} + |\alpha|^{n_1} \right) = C_2 n^\sigma|\alpha|^{n} \left(1 + \frac{1}{|\alpha|^{n-n_1}} \right)\\
& \leqslant C_2 n^\sigma|\alpha|^{n} \left(1 + \frac{1}{|\alpha|} \right) = C_5n^\sigma |\alpha|^{n}
\end{split}
\end{equation*}
and
\begin{equation*} 
\begin{split}
| U_n-U_{n_1}| & \geqslant |U_n| - |U_{n_1}| \geqslant C_1n^\sigma |\alpha|^{n} - C_2 n^\sigma|\alpha|^{n_1}   = C_1 n^\sigma|\alpha|^{n} \left(1 - \frac{C_2}{C_1|\alpha|^{n-n_1}} \right)\\
& \geqslant C_1 n^\sigma|\alpha|^{n} \left(1 - \frac{C_2}{C_1|\alpha|} \right) = C_6n^\sigma |\alpha|^{n}.
\end{split}
\end{equation*}
Similarly, we have
\begin{equation*}
\begin{split}
| V_m-V_{m_1}| & \leqslant |V_m| + |V_{m_1}| \leqslant C_4 m^\tau\left(|\beta|^{m} + |\beta|^{m_1} \right) = C_4 m^\tau|\beta|^{m} \left(1 + \frac{1}{|\beta|^{m-m_1}} \right)\\
& \leqslant C_4 m^\tau |\beta|^{m} \left(1 + \frac{1}{|\beta|} \right) = C_7 m^\tau |\beta|^{m}
\end{split}
\end{equation*}
and
\begin{equation*}
\begin{split}
| V_m-V_{m_1}| & \geqslant |V_m| - |V_{m_1}| \geqslant C_3 m^\tau|\beta|^{m} - C_4 m^\tau |\beta|^{m_1}\\
& = C_3 m^\tau|\beta|^{m} \left(1 - \frac{C_4}{C_3|\beta|^{m-m_1}} \right)\\
& \geqslant C_3 m^\tau|\beta|^{m} \left(1 - \frac{C_4}{C_3|\beta|} \right) = C_8 m^\tau|\beta|^{m}.
\end{split}
\end{equation*}
Therefore, we have 
\begin{equation}\label{ineq1}
C_6 n^\sigma|\alpha|^n \leqslant | U_n-U_{n_1}| = |V_m - V_{m_1}| \leqslant C_7 m^\tau|\beta|^{m}
\end{equation}
and
\begin{equation}\label{ineq2}
C_5 n^\sigma|\alpha|^n  \geqslant  |U_n-U_{n_1} | = | V_m - V_{m_1}| \geqslant C_8 m^\tau|\beta|^m.
\end{equation}
Note that we proved \eqref{ineq1} and \eqref{ineq2} only under the assumption that $n>n_1\geq N_2$ and $m>m_1\geq M_2$. However since by assumption $n>n_1\geqslant N_0$
and $m>m_1\geqslant M_0$ we have $|U_n|>|U_{n_1}|$ and $|V_m|>|V_{m_1}|$ respectively. Therefore by enlarging $C_7$ and $C_5$ respectively decreasing $C_6$ and $C_8$ we obtain
that \eqref{ineq1} and \eqref{ineq2} also holds under the assumption that $n\geq N_2$, $n_1\geq N_0$ and $m\geq M_2$, $m_1\geq M_0$.
Thus
\begin{equation}\label{computer}
n \leqslant m \frac{\log |\beta|}{\log |\alpha|} + \tau\frac{\log m}{\log |\alpha|} + C_9,
\end{equation}
where $ 0 < \frac{\log |\beta|}{\log |\alpha|} < 1 $.

Inequality \eqref{computer} implies that $m > n$ for $m \geqslant M_3$, where $M_3$ is sufficiently large. 
Denote by $N_3$ the infimum for $n$ when $m \geqslant M_3$.
Let us assume in the following that $n> N_4 = \max \lbrace N_0, N_1,N_2,N_3,2 \rbrace$ and $m > M_4 = \max \lbrace M_0, M_1,M_2,M_3,2 \rbrace$ (and $n_1\geq N_0$ and $m_1\geq M_0$).
Let us note that if $m$ is bounded from above by an effectively computable constant as $M_4$ also $n$ is bounded from above by an effective computable constant due to inequality \eqref{computer}. Thus we can
deduce that also $c$ is bounded and Theorem \ref{Main} holds in this case. Note that we assume for technical reasons that $N_4,M_4\geq 2$.
Therefore we may assume that $m>M_4$ and hence $m > n$. 

Furthermore let us fix the following notation for the rest of the paper. Let us write $K=\bbQ(\alpha_1,\dots,\alpha_t,\beta_1,\dots,\beta_s)$ and $d=[K:\bbQ]$.


\subsection{Linear forms in logarithms}

We refer to equation \eqref{eqt} and make use of the asymptotic estimates \eqref{eq:asymp1} and \eqref{eq:asymp2}. Thus we get
\begin{multline*}
\left( a(n) \alpha^n + L(a'\alpha'^n ) \right) - \left( a(n_1) \alpha^{n_1} + L(a'\alpha'^{n_1}  )\right) =\\
\left( b(m) \beta^m + L(b'\beta'^m  )\right) - \left( b(m_1) \beta^{m_1} + L( b'\beta'^{m_1} ) \right)
\end{multline*}
Collecting the ``large'' terms on the left hand side of the equation we obtain
$$
a(n) \alpha^n - b(m)\beta^m=a(n_1) \alpha^{n_1} - b(m_1) \beta^{m_1} + L\left(a' {\alpha'}^n + a' {\alpha'}^{n_1} + b' {\beta'}^m + b' {\beta'}^{m_1} \right)
$$
and therefore the inequality
$$
\vert a(n) \alpha^n - b(m) \beta^m \vert \leqslant |a(n_1)||\alpha|^{n_1}  + |b(m_1)||\beta|^{m_1} + a'{\alpha'}^n + a'{\alpha'}^{n_1} +b'{\beta'}^m + b'{\beta'}^{m_1}.
$$ 
Dividing through $b(m) \beta^m$ and using the inequalities \eqref{ineq1} and \eqref{ineq2}, we get (note that we assume $n\geq N_3$ and $m\geq M_3$, i.e. $|a(n)|\geq |a(n_1)|$ and $|b(m)|\geq |b(m_1)|$):
\begin{align*}
\left| \frac{a(n) \alpha^n}{b(m) \beta^m}  - 1 \right| 
& \leqslant \frac{|a(n_1)||\alpha|^{n_1}}{|b(m)||\beta|^m}+\frac{|b(m_1)||\beta|^{m_1}}{|b(m)||\beta|^m}+\frac{a'{\alpha'}^n}{|b(m)||\beta|^m}+\frac{a'{\alpha'}^{n_1}}{|b(m)||\beta|^m}\\
&\quad\quad +\frac{b'{\beta'}^m}{|b(m)||\beta|^m}+\frac{b'{\beta'}^{m_1}}{|b(m)||\beta|^m} \\
& \leqslant \frac{C_7m^\tau |a(n_1)||\alpha|^{n_1}}{C_6n^\sigma |b(m)||\alpha|^n}+\frac{|b(m_1)||\beta|^{m_1}}{|b(m)||\beta|^m}+
\frac{C_7m^\tau a'{\alpha'}^n}{C_6 n^\sigma |b(m)||\alpha|^n}\\
&\quad\quad +\frac{C_7m^\tau a' {\alpha'}^{n_1}}{C_6n^\sigma |b(m)||\alpha|^n}+\frac{b'{\beta'}^m}{|b(m)||\beta|^m}+\frac{b'{\beta'}^{m_1}}{|b(m)||\beta|^m} \\
& \leqslant C_{11} |\alpha|^{n_1-n} +  |\beta|^{m_1-m}  + C_{12} \left(\frac{ |\alpha|}{ \alpha'} \right)^{-n} + C_{13} |\alpha|^{n_1-n} \\
& \quad\quad + C_{14} \left( \frac{|\beta|}{\beta'} \right)^{-m}+ C_{15}|\beta|^{m_1-m} \\
& \leqslant C_{11} |\alpha|^{n_1-n} +|\beta|^{m_1-m}  + C_{12} \left(\frac{|\alpha|}{\alpha'} \right)^{n_1-n} + C_{13} |\alpha|^{n_1-n} \\
& \quad\quad + C_{14} \left( \frac{|\beta|}{\beta'} \right)^{m_1-m}+ C_{15}|\beta|^{m_1-m} \\
& \leqslant \max \left\{ C_{16} \left(\frac{|\alpha|}{\alpha'}\right)^{n_1-n} , C_{17}  \left(\frac{|\beta|}{\beta'}\right)^{m_1-m}  \right\}.
\end{align*} 
Note that $\frac{m^\tau |a(n_1)|}{n^\sigma |b(m)|} \frac{m^\tau |a(n)|}{n^\sigma |b(m)|}$, $\frac{|b(m_1)|}{|b(m)|}$ and so on are bounded by absolute
constants since $\deg(a)=\sigma$ and $\deg(b)=\tau$. Hence we obtain the inequality
\begin{align} \label{Case0}
\left| \frac{a(n)}{b(m) } \alpha^n \beta^{-m}- 1 \right|
& \leqslant \max \left\{ C_{16} \left(\frac{|\alpha|}{\alpha'}\right)^{n_1-n} , C_{17} \left( \frac{|\beta|}{\beta'}\right)^{m_1-m}  \right\}.
\end{align} 

Let us introduce
$${\it\Lambda} = n\log |\alpha| - m \log |\beta| + \log \left|\frac{a(n)}{b(m)}\right|$$
and assume that $|{\it\Lambda}| \leqslant 0.5$ and $\frac{a(n)}{b(m) } \alpha^n \beta^{-m}>0$.
Further, we put
$${\it\Phi} = e^{\it\Lambda}-1 = \left|\frac{a(n) }{b(m)}\right| |\alpha|^n |\beta|^{-m}- 1$$
and use the theorem of Baker and W\"ustholz (Theorem \ref{BaWu}) with the data
\[k=3, \quad \eta_1 = \left|\frac{a(n)}{b(m)}\right| , \quad b_1 = 1, \quad \eta_2 = |\alpha|, \quad b_2 = n, \quad \eta_3 = |\beta|,  \quad b_3 = -m. \]
Note that with this data we have $B=m$. It should be noted that we have complete information on the minimal polynomial of $\alpha$ and $\beta$. 
Therefore, $h'(\alpha)$, $h'(\beta)$ are effectively computable. Moreover, due to Lemma \ref{lem:po-height} we have $h_0\left(\frac{a(n)}{b(m)}\right)\leqslant \widetilde C \log m$ and thus
\[ h' \left(\frac{a(n)}{b(m)}\right) = \dfrac{1}{d}  \max \left\{d  h_0\left(\frac{a(n)}{b(m)}\right), \left| \log \left( \frac{a(n)}{b(m)}\right) \right|, 1 \right\}  \leqslant \widetilde{C'} \log m. \]

Before we can apply Theorem \ref{BaWu} we have to ensure that ${\it\Phi} \neq 0$. Assume to the contrary that ${\it\Phi} = 0$,
then $\frac{a(n)}{b(m)} = \pm \frac{\beta^m}{\alpha^n}$. With the use of Lemma \ref{height} we get
$$
\widetilde{C} \log m \geqslant h_0\left( \frac{a(n)}{b(m)} \right) = h_0\left( \frac{\beta^m}{\alpha^n}\right) \geqslant C_0 \max\{n,m\}=C_0m
$$
which yields an absolute upper bound for $m$. Therefore also $n$ and $c$ are bounded, i.e. Theorem \ref{Main} holds in this special case.

An application of Theorem \ref{BaWu} yields
\begin{equation*}
\log |{\it\Phi}| \geqslant -C(3,d) h'\left(\frac{a(n)}{b(m)}\right) h'(\alpha) h'(\beta) \log m - \log 2
\end{equation*}
and together with inequality \eqref{Case0} we have
\[\min \left\{ (n-n_1)\log\left(\frac{|\alpha|}{\alpha'}\right),(m-m_1)\log\left(\frac{|\beta|}{\beta'}\right)\right\} < C_{18} (\log m)^2. \]

Thus we have proved so far:
\begin{lemma}\label{lem:Case0}
 Assume that $(n,m,n_1,m_1)$ is a solution to equation \eqref{eqt} with $m>m_1$. Then we have
 $$\min \left\{ (n-n_1)\log\left(\frac{|\alpha|}{\alpha'}\right),(m-m_1)\log\left(\frac{|\beta|}{\beta'}\right) \right\} < C_{18} (\log m)^2.$$
\end{lemma}

Note that in the case that $|{\it\Lambda}| > 0.5$ or $\frac{a(n)}{b(m) } \alpha^n \beta^{-m}<0$ inequality \eqref{Case0} is possible only if 
\begin{align*} 
\max \left\{ C_{16}\left(\frac{|\alpha|}{\alpha'}\right)^{n_1-n} ,C_{17} \left(\frac{|\beta|}{\beta'}\right)^{m_1-m}\right\} \geqslant e^{\frac{1}{2}} - 1 > 0.648,
\end{align*} 
which leads to either 
$$n-n_1 \leqslant \frac{\log \left(\frac{C_{16}}{0.648}\right)}{\log \left(\frac{|\alpha|}{\alpha'}\right) }$$
or 
$$m-m_1 \leqslant \frac{\log \left(\frac{C_{17}}{0.648}\right)}{\log \left(\frac{|\beta|}{\beta'}\right) }.$$ 
These can be covered by the bound provided by Lemma \ref{lem:Case0} as long as we choose 
\begin{align*} 
C_{18} \geqslant \frac{1}{(\log M_3)^2} \max \left\{ \log \left(\frac{C_{16}}{0.648}\right),
 \log \left(\frac{C_{17}}{0.648}\right)\right\}.
\end{align*}

Now we have to distinguish between the following two cases:

\noindent\textbf{Case 1.} Let us assume that
$$
\min \left\{(n-n_1)\log \left(\frac{|\alpha|}{\alpha'}\right),(m-m_1)\log\left(\frac{|\beta|}{\beta'}\right) \right\}
= (n-n_1) \log \left(\frac{|\alpha|}{\alpha'}\right),
$$ 
i.e. we assume that $\left(\frac{|\alpha|}{\alpha'}\right)^{n_1-n}\leqslant \left(\frac{|\beta|}{\beta'}\right)^{m_1-m}$.

By collecting the ``large terms" on the left hand side, we can rewrite equation~\eqref{eqt} as
$$
a(n) \alpha^n - a(n_1) \alpha^{n_1} - b(m) \beta^m  = - b(m_1) \beta^{m_1} + L\left(a' {\alpha'}^n + a'{\alpha'}^{n_1} + b'{\beta'}^m + b'{\beta'}^{m_1} \right)
$$
and obtain the inequality
\begin{multline*}
\left| a(n) \alpha^{n_1}\left(\alpha^{n-n_1} -\frac{a(n_1)}{a(n)}\right) - b(m) \beta^m \right|\\ \leqslant  |b(m_1)||\beta|^{m_1} + a'{\alpha'}^n + a'{\alpha'}^{n_1} + b'{\beta'}^m + b'{\beta'}^{m_1}.
\end{multline*}
Dividing through $b(m) \beta^m$ and using the inequalities \eqref{ineq1} and \eqref{ineq2}, we get
\begin{align*}
\left| \frac{a(n) \alpha^{n_1}\left(\alpha^{n-n_1} -\frac{a(n_1)}{a(n)}\right)}{b(m) \beta^m}  - 1 \right|& 
\leqslant \frac{|b(m_1)||\beta|^{m_1}}{|b(m)||\beta|^{m}} + \frac{a' {\alpha'}^n}{|b(m)||\beta|^m}  + \frac{a' {\alpha'}^{n_1}}{|b(m)||\beta|^m}\\
&\quad\quad + \frac{b'{\beta'}^m}{|b(m)||\beta|^m} + \frac{b'{\beta'}^{m_1}}{|b(m)||\beta|^m} 
\end{align*}
where
\begin{align*}
\frac{|b(m_1)||\beta|^{m_1}}{|b(m)||\beta|^{m}} 
& \leqslant \left(\frac{|\beta|}{\beta'}\right)^{m_1-m},\\
\frac{a' {\alpha'}^n}{|b(m)||\beta|^m} 
& \leqslant \frac{C_7m^\tau a'{\alpha'}^n}{C_6n^\sigma |b(m)||\alpha|^n} 
\leqslant C_{19} \left(\frac{|\alpha|}{\alpha'} \right)^{-n}\\
& \leqslant C_{19} \left(\frac{|\alpha|}{\alpha'} \right)^{n_1-n}
\leqslant C_{19} \left(\frac{|\beta|}{\beta'}\right)^{m_1-m}, \\
\frac{a' {\alpha'}^{n_1}}{|b(m)| |\beta|^m} 
& \leqslant  \frac{C_7m^\tau a' {\alpha'}^{n_1}}{C_6n^\sigma |b(m)| |\alpha|^n} 
\leqslant  C_{20} \left(\frac{|\alpha|}{\alpha'}\right)^{n_1-n} \leqslant C_{20} \left(\frac{|\beta|}{\beta'}\right)^{m_1-m},\\
\frac{b' {\beta'}^m}{|b(m)| |\beta|^m}
& \leqslant C_{21} \left(\frac{|\beta|}{\beta'} \right)^{-m} 
\leqslant C_{21} \left(\frac{|\beta|}{\beta'} \right)^{m_1-m}\\
\frac{b'{\beta'}^{m_1}}{|b(m)| |\beta|^m}
& \leqslant C_{22}|\beta|^{m_1-m} 
\leqslant C_{22} \left(\frac{|\beta|}{\beta'}\right)^{m_1-m}.
\end{align*}
Hence we obtain the inequality
\begin{equation} \label{Case1}
\left|\frac{a(n) }{b(m) }\left(\alpha^{n-n_1} -\frac{a(n_1)}{a(n)}\right)  \alpha^{n_1} \beta^{-m} - 1 \right| \leqslant C_{23} \left(\frac{|\beta|}{\beta'}\right)^{m_1-m}.
\end{equation} 

\noindent\textbf{Case 2.} Let us assume that
$$
\min \left\{ (n-n_1)\log \left(\frac{|\alpha|}{\alpha'}\right),(m-m_1)\log\left(\frac{|\beta|}{\beta'}\right) \right\} = (m-m_1)\log\left(\frac{|\beta|}{\beta'}\right).
$$ 
i.e. we assume that $\left(\frac{|\beta|}{\beta'}\right)^{m_1-m}\leqslant \left(\frac{|\alpha|}{\alpha'}\right)^{n_1-n}$.

Similarly as in Case 1 we collect the ``large terms" on the left hand side and rewrite equation \eqref{eqt} as
$$
a(n) \alpha^n - b(m) \beta^m + b(m_1) \beta^{m_1}  = - a(n_1) \alpha^{n_1} + L\left(a' {\alpha'}^n + a'{\alpha'}^{n_1} +b'{\beta'}^m + b'{\beta'}^{m_1} \right)
$$
and obtain the inequality
\begin{multline*}
\left| b(m) \beta^{m_1}\left(\beta^{m-m_1} -\dfrac{b(m_1)}{b(m)}\right) - a(n) \alpha^n \right| \\ \leqslant |a(n_1)||\alpha|^{n_1}  + a'{\alpha'}^n + a'{\alpha'}^{n_1} + b'{\beta'}^m + b'{\beta'}^{m_1}.
\end{multline*}
We obtain the inequality
\begin{equation} \label{Case2}
\left| \frac{b(m)}{a(n)}\left(\beta^{m-m_1} -\frac{b(m_1)}{b(m)}\right) \alpha^{-n} \beta^{m_1} - 1 \right| 
\leqslant C_{28} \left(\frac{|\alpha|}{\alpha'}\right)^{n_1-n}
\end{equation} 
by the same arguments as in Case 1 by interchanging $a(n), \alpha, n, n_1, a'$ and $\alpha'$ with $b(m), \beta, m, m_1, b'$ and $\beta'$.


We want to apply Theorem \ref{BaWu} to both inequalities \eqref{Case1} and \eqref{Case2} respectively.
Let us consider the first case more closely. We write
$${\it\Lambda}_1 =n_1 \log |\alpha| - m \log |\beta| + \log \left| \frac{a(n)}{b(m)}\left(\alpha^{n-n_1} -\frac{a(n_1)}{a(n)}\right)  \right|$$
and assume that $|{\it\Lambda}_1| \leqslant 0.5$ and $\frac{a(n)}{b(m)}\left(\alpha^{n-n_1} -\frac{a(n_1)}{a(n)}\right)>0$. Further, we put
$${\it\Phi}_1 = e^{{\it\Lambda}_1}-1= \left|\frac{a(n)}{b(m)}\left(\alpha^{n-n_1} -\frac{a(n_1)}{a(n)}\right)\right|  |\alpha|^{n_1} |\beta|^{-m} - 1$$
and aim to apply Theorem \ref{BaWu} with $B = m$. Further, we have
\begin{gather*}
\eta_1= \left|\frac{a(n) }{b(m) }\left(\alpha^{n-n_1} -\frac{a(n_1)}{a(n)}\right)\right| , \quad b_1 = 1, \\
\eta_2 =  |\alpha|, \quad b_2 = n_1 , \quad \eta_3 = |\beta|, \quad b_3=-m.
\end{gather*}
It should be noted that as before $h'(\alpha)$ and $h'(\beta)$ are effectively computable.
For $h'(\eta_1)$, we can use the properties of height and the results of Lemma~\ref{lem:Case0} and Lemma~\ref{lem:po-height} to get 
\begin{align*}
h_0(\eta_1) &= h_0\left(\frac{a(n)}{b(m)}\left(\alpha^{n-n_1} -\frac{a(n_1)}{a(n)}\right) \right) \\
& \leqslant h_0\left(\frac{a(n)}{b(m)}\right) + (n-n_1) h_0(\alpha )  + h_0\left(\frac{a(n_1)}{a(n)}\right)+\log 2\\
& \leqslant h_0\left(\frac{a(n)}{b(m)}\right) + \frac{C_{18} (\log m)^2 }{\log \left(\frac{|\alpha|}{\alpha'}\right)} h_0(\alpha )  + h_0\left(\frac{a(n_1)}{a(n)}\right)+ \log 2\\
& \leqslant C_{29} (\log m)^2 
\end{align*}
and thus
$$h'(\eta_1) = \frac{1}{d} \max \left\{ dh_0(\eta_1), |\log\eta_1|, 1 \right\} \leqslant C_{30} (\log m)^2. $$

Now let us turn to the second case. We write
$${\it\Lambda}_2 =  m_1 \log |\beta| - n \log |\alpha| + \log \left| \frac{b(m)}{a(n)}\left(\beta^{m-m_1} -\frac{b(m_1)}{b(m)}\right) \right|$$
and assume that $|{\it\Lambda}_2| \leqslant 0.5$ and $\frac{b(m)}{a(n)}\left(\beta^{m-m_1} -\frac{b(m_1)}{b(m)}\right)>0$. Further, we put
$${\it\Phi}_2 = e^{{\it\Lambda}_2}-1 = \left|\frac{b(m)}{a(n)}\left(\beta^{m-m_1} -\frac{b(m_1)}{b(m)}\right) \right| |\alpha|^{-n} |\beta|^{m_1} - 1 $$
and aim to apply Theorem \ref{BaWu}.
As in the previous case we also have $B = m$. Further, we have
\begin{gather*}
 \eta_1= \left|\frac{b(m)}{a(n)}\left(\beta^{m-m_1} -\frac{b(m_1)}{b(m)}\right)\right|, \quad b_1 = 1,\\
 \eta_2 = |\alpha|, \quad b_2 = -n , \quad \eta_3 = |\beta|, \quad b_3=m_1.
\end{gather*} 
It should be noted that as before $h'(\alpha)$ and $h'(\beta)$ are effectively computable.
For $h'(\eta_1)$, we can use the properties of height and the results of Lemma~\ref{lem:Case0} and Lemma~\ref{lem:po-height} to get 
\begin{align*}
h_0(\eta_1) &= h_0\left(\frac{b(m)}{a(n)}\left(\beta^{m-m_1} -\frac{b(m_1)}{b(m)}\right) \right)\\
& \leqslant h_0\left(\frac{b(m)}{a(n)}\right) + (m-m_1) h_0(\beta )  +h_0\left(\frac{b(m_1)}{b(m)}\right) +  \log 2\\
& \leqslant h_0\left(\frac{b(m)}{a(n)}\right) + \frac{C_{18} (\log m)^2 }{ \log \left(\frac{|\beta|}{\beta'}\right)} h_0(\beta )+ h_0\left(\frac{b(m_1)}{b(m)}\right) + \log 2\\
& \leqslant C_{31} (\log m)^2  
\end{align*}
and thus 
$$h'(\eta_1) = \frac{1}{d} \max \left\{ dh_0(\eta_1), |\log\eta_1|, 1 \right\} \leqslant C_{32} (\log m)^2.$$

Before we can apply Theorem \ref{BaWu} we have to ensure that ${\it\Phi}_i \neq 0$ for $i=1, 2$.
Firstly we deal with the assumption that ${\it\Phi}_1 = 0$,
i.e. $\pm \frac{a(n)}{b(m)}\left(\alpha^{n-n_1} -\frac{a(n_1)}{a(n)}\right) = \frac{\beta^{m}}{\alpha^{n_1}} $. This together with Lemma \ref{lem:Case0} yields
$$
h_0 \left( \frac{\beta^{m}}{\alpha^{n}} \right) = h_0 \left(\frac{a(n) }{b(m) }\left(\alpha^{n-n_1} -\frac{a(n_1)}{a(n)}\right) \right) < C_{29} (\log m)^2
$$
as determined before. With the use of Lemma \ref{height} we get
\begin{align*}
C_{29} (\log m)^2 > h_0 \left( \frac{\beta^{m}}{\alpha^{n_1}} \right) \geqslant C_0 \max \{n_1, m \} \geqslant C_0 m .
\end{align*}
Thus $m$ is bounded by an effectively computable constant. Besides, since $m > n$ so $n$ is also bounded and therefore also $c$, i.e. Theorem \ref{Main}
holds in this case. A similar argument also applies to Case 2.

Now, we are ready to apply Theorem \ref{BaWu} and get
\begin{align*}
\log |{\it\Phi}_i|  > & -C(3,d) h'(\eta_1) h'(\alpha) h'(\beta) \log m - \log 2
\end{align*}
for $i = 1,2$. Combining this inequality with the inequalities \eqref{Case1} and \eqref{Case2}, we obtain
\[ (m-m_1)\log \left(\frac{|\beta|}{\beta'}\right) < C_{33} (\log m)^3 \quad 
\mbox{and}
\quad
(n-n_1)\log \left(\frac{|\alpha|}{\alpha'}\right) < C_{34} (\log m)^3 \]
respectively. Let $C_{35} = \max \lbrace C_{33}, C_{34} \rbrace$. These two inequalities yield together with Lemma \ref{lem:Case0} the following lemma:

\begin{lemma}\label{lem:Case1-2}
 Assume that $(n,m,n_1,m_1)$ is a solution to equation \eqref{eqt} with $m>m_1$. Then we have
 $$\max \left\{ (n-n_1)\log \left( \frac{|\alpha|}{\alpha'}\right),(m-m_1)\log\left(\frac{|\beta|}{\beta'}\right) \right\} < C_{35} (\log m)^3.$$
\end{lemma}

Note that in view of $|{\it\Lambda}_1| > 0.5$ or $\frac{a(n)}{b(m)}\left(\alpha^{n-n_1} -\frac{a(n_1)}{a(n)}\right)<0$, inequality \eqref{Case1} is possible only if 
$$ 
C_{23} \left(\frac{|\beta|}{\beta'}\right)^{m_1-m} \geqslant e^{\frac{1}{2}} - 1 > 0.648,
$$ 
which leads to $m-m_1 \leqslant \frac{\log \left(\frac{C_{23}}{0.648}\right)}{\log\left(\frac{|\beta|}{\beta'}\right) }$.
In view of $|{\it\Lambda}_2| > 0.5$ or $\frac{b(m)}{a(n)}\left(\beta^{m-m_1} -\frac{b(m_1)}{b(m)}\right)<0$, inequality \eqref{Case2} is possible only if
$$ 
C_{28} \left(\frac{|\alpha|}{\alpha'}\right)^{n_1-n} \geqslant e^{\frac{1}{2}} - 1 > 0.648,
$$ 
which leads to $n-n_1 \leqslant \frac{\log \left(\frac{C_{28}}{0.648}\right)}{\log\left(\frac{|\alpha|}{\alpha'}\right)}$. 
Both cases can be covered by the bound provided by Lemma \ref{lem:Case1-2} as long as 
$$
C_{35} \geqslant \frac{1}{(\log M_3)^3} \max \left\{  \log \left(\frac{C_{28}}{0.648}\right),
  \log \left(\frac{C_{23}}{0.648}\right)  \right\}.
$$

One more time we have to apply Theorem \ref{BaWu}. This time we rewrite equation~\eqref{eqt} by collecting ``large" terms on the left hand side as
$$
a(n) \alpha^n - a(n_1) \alpha^{n_1} - b(m) \beta^m + b(m_1) \beta^{m_1} = L\left(a'{\alpha'}^n + a'{\alpha'}^{n_1} + b'{\beta'}^m + b'{\beta'}^{m_1} \right)
$$
and obtain
\begin{multline*}
\left|  a(n) \alpha^{n_1}\left(\alpha^{n-n_1} -\frac{a(n_1)}{a(n)}\right) - b(m) \beta^{m_1} \left(\beta^{m-m_1}  -\frac{b(m_1)}{b(m)}\right)  \right|\\
\leqslant a'{\alpha'}^n + a'{\alpha'}^{n_1} + b'{\beta'}^m + b'{\beta'}^{m_1}.
\end{multline*}
Dividing through $b(m) \beta^{m_1} \left(\beta^{m-m_1}  -\frac{b(m_1)}{b(m)}\right) $ and using the inequalities \eqref{ineq1} and \eqref{ineq2} we get
\begin{align*}
\left| \frac{a(n) \alpha^{n_1}\left(\alpha^{n-n_1} -\frac{a(n_1)}{a(n)}\right)}{b(m) \beta^{m_1} \left(\beta^{m-m_1}  -\frac{b(m_1)}{b(m)}\right)}  - 1 \right| & 
\leqslant   \frac{a'{\alpha'}^n}{|b(m)||\beta|^{m_1} \left|\beta^{m-m_1}  -\frac{b(m_1)}{b(m)}\right|}\\ 
+ \frac{a'{\alpha'}^{n_1}}{|b(m)||\beta|^{m_1}\left|\beta^{m-m_1}  -\frac{b(m_1)}{b(m)}\right|} & + \frac{b'{\beta'}^m }{|b(m)||\beta|^{m_1} \left|\beta^{m-m_1}  -\frac{b(m_1)}{b(m)}\right|}\\
&+ \frac{b'{\beta'}^{m_1}}{|b(m)||\beta|^{m_1}\left|\beta^{m-m_1}  -\frac{b(m_1)}{b(m)}\right|}.
\end{align*}
We make use of inequality \eqref{computer} to get
\begin{align*}
{\alpha'}^n & = \exp\left(n \log \alpha' \right)\\
& < \exp\left(m \frac{\log|\beta|}{\log |\alpha|}\log \alpha' + \tau\frac{\log \alpha'}{\log |\alpha|} \log m + C_9 \log \alpha' \right)\\
& = \exp\left( m \frac{\log \alpha' }{ \log |\alpha|} \log |\beta|  \right)\left(m^\tau\right)^{\frac{\log \alpha'}{\log |\alpha|}} \exp \left( C_9 \log \alpha' \right)\\
& < C_{36} m^\tau \gamma^m , 
\end{align*}
where $\gamma = |\beta|^{\frac{\log \alpha' }{ \log |\alpha|}}$. Note that since $|\alpha|>\alpha'>1$ and $|\beta|>1$ we have that $|\beta|>\gamma>1$. So that
\begin{align*}
\frac{a' {\alpha'}^n}{|b(m)||\beta|^{m_1} \left|\beta^{m-m_1}  -\frac{b(m_1)}{b(m)}\right|}  
& < \frac{C_{36}a'm^\tau\gamma^m}{|b(m)||\beta|^{m}\left|1 - \beta^{m_1-m}\frac{b(m_1)}{b(m)}\right|} \\
& = \frac{C_{36}a'm^\tau\gamma^m}{|b(m)||\beta|^{m}\left| 1 - \frac{b(m_1)}{b(m)\beta^{m-m_1}} \right|} \\
&  \leqslant   \frac{C_{36}a'm^\tau\gamma^m}{|b(m)||\beta|^{m}\left|1 - \frac{1}{\beta} \right|}
\leqslant C_{37} \left( \frac{|\beta|}{\gamma} \right)^{-m}.
\end{align*}
In addition, since we assume that $\alpha'> 1$, we have
\begin{align*}
\frac{a'{\alpha'}^{n_1}}{|b(m)||\beta|^{m_1} \left|\beta^{m-m_1}  -\frac{b(m_1)}{b(m)}\right|} 
& < \frac{a' {\alpha'}^{n}}{|b(m)||\beta|^{m} \left|1 - \frac{b(m_1)}{b(m)\beta^{m-m_1}}\right|}\\ 
& <  \frac{C_{36}a' m^\tau \gamma^{m}}{|b(m)||\beta|^{m}\left| 1 - \frac{1}{\beta} \right|} 
\leqslant C_{37} \left( \frac{|\beta|}{\gamma} \right)^{-m}.
\end{align*}
Furthermore,
$$
\frac{b'{\beta'}^m }{|b(m)||\beta|^{m_1} \left|\beta^{m-m_1}  -\frac{b(m_1)}{b(m)}\right|}
\leqslant \frac{b'{\beta'}^m }{|b(m)||\beta|^{m}\left| 1 - \frac{1}{\beta} \right|} 
\leqslant C_{38} \left( \frac{|\beta|}{\beta'} \right)^{-m}.
$$  
Since we may assume that $\beta'> 1$ we get
$$
\frac{b'{\beta'}^{m_1} }{|b(m)||\beta|^{m_1} \left|\beta^{m-m_1}  - \frac{b(m_1)}{b(m)} \right|} 
\leqslant  \frac{b'{\beta'}^{m} }{|b(m)||\beta|^{m}\left| 1 - \frac{1}{\beta} \right|} 
= C_{39} \left( \frac{|\beta|}{\beta'} \right)^{-m}.
$$
Therefore,
\begin{align}  \label{Case3}
\left| \frac{a(n) \left(\alpha^{n-n_1} -\frac{a(n_1)}{a(n)}\right)}{b(m) \left(\beta^{m-m_1}  -\frac{b(m_1)}{b(m)}\right)}\alpha^{n_1}\beta^{-m_1}  - 1 \right|  
\leqslant  C_{40} \Gamma^{-m},
\end{align}
where $\Gamma=\min \left\{ \frac{  |\beta| }{ \beta'}, \frac{  |\beta| }{ \gamma} \right\}$.
In this final step we consider the linear form
$${\it\Lambda}_3 =n_1 \log |\alpha| - m_1 \log |\beta| + \log \left| \frac{a(n) \left(\alpha^{n-n_1} -\frac{a(n_1)}{a(n)}\right)}{b(m) \left(\beta^{m-m_1}  -\frac{b(m_1)}{b(m)}\right)}\right|$$
and assume that $|{\it\Lambda}_3| \leqslant 0.5$ and $\frac{a(n) \left(\alpha^{n-n_1} -\frac{a(n_1)}{a(n)}\right)}{b(m) \left(\beta^{m-m_1}  -\frac{b(m_1)}{b(m)}\right)}>0$. Further, we put
$${\it\Phi}_3 = e^{{\it\Lambda}_3} - 1 =\left|\frac{a(n) \left(\alpha^{n-n_1} -\frac{a(n_1)}{a(n)}\right)}{b(m) \left(\beta^{m-m_1}  -\frac{b(m_1)}{b(m)}\right)}\right||\alpha|^{n_1}|\beta|^{-m_1} - 1.$$

As before we take $B=m$ and we choose
\begin{gather*}
 \eta_1 =\left|\frac{a(n) \left(\alpha^{n-n_1} -\frac{a(n_1)}{a(n)}\right)}{b(m) \left(\beta^{m-m_1}  -\frac{b(m_1)}{b(m)}\right)}\right|, \quad b_1 = 1,\\
 \eta_2 = |\alpha|, \quad b_2 = n_1, \quad \eta_3 = |\beta|, \quad b_3 = -m_1.
\end{gather*} 

For $h'(\eta_1)$, we can use the properties of the height and the results of Lemma \ref{lem:po-height} and Lemma \ref{lem:Case1-2} to get 
\begin{align*}
h_0(\eta_1) &= h_0\left(\frac{a(n) \left(\alpha^{n-n_1} -\frac{a(n_1)}{a(n)}\right)}{b(m) \left(\beta^{m-m_1}  -\frac{b(m_1)}{b(m)}\right)} \right) \\
& \leqslant h_0\left(\frac{a(n)}{b(m)}\right) + (n-n_1) h_0(\alpha ) + (m-m_1) h_0(\beta )\\
& \quad+ h_0\left(\frac{a(n_1)}{a(n)}\right)+h_0\left(\frac{b(m_1)}{b(m)}\right)+ 2 \log 2\\
& \leqslant h_0\left(\frac{a(n)}{b(m)}\right) +   \frac{C_{35} h_0(\alpha ) (\log m)^3 }{\log \left(\frac{|\alpha|}{\alpha'}\right)} 
+ \frac{C_{35}h_0(\beta ) (\log m)^3}{\log \left(\frac{|\beta|}{\beta'}\right) }\\
&\quad +h_0\left(\frac{a(n_1)}{a(n)}\right)+h_0\left(\frac{b(m_1)}{b(m)}\right)+  2 \log 2\\
& \leqslant C_{41} (\log m)^3  
\end{align*}
and thus 
$$h'(\eta_1) = \frac{1}{d} \max \left\{ dh_0(\eta_1), |\log\eta_1|, 1 \right\} \leqslant C_{42} (\log m)^3. $$
It should be noted that as before $h'(\alpha)$ and $h'(\beta)$ are effectively computable. 

Before we can apply Theorem \ref{BaWu} we have to ensure that ${\it\Phi}_3 \neq 0$,
i.e. 
$$ \pm \frac{a(n) \left(1- \frac{a(n_1)\alpha^{n_1-n}}{a(n)} \right)}{b(m)\left(1- \frac{b(m_1)\beta^{m_1-m}}{b(m)}\right)} = \frac{\beta^m}{\alpha^n}.$$
This together with Lemma \ref{lem:Case1-2} yields
$$
h_0 \left( \frac{\beta^{m}}{\alpha^{n}} \right) = h_0 \left( \frac{a(n) \left(1- \frac{a(n_1)\alpha^{n_1-n}}{a(n)} \right)}{b(m)\left(1- \frac{b(m_1)\beta^{m_1-m}}{b(m)}\right)} \right) < C_{43} (\log m)^3.
$$
Similar to the argument in Case 1 and Case 2, we deduce by using Lemma \ref{height} that
$$
C_{43} (\log m)^3 > h_0 \left( \frac{\beta^{m}}{\alpha^{n}} \right) \geqslant C_0 \max \{n, m \} \geqslant C_0 m .
$$
Thus $m$ is bounded by an effectively computable constant. Besides, since $m > n$ so $n$ is also bounded and therefore also $c$ and
we deduce Theorem \ref{Main} in this case.

Now an application of Theorem \ref{BaWu} yields
$$
\log |{\it\Phi}_3|
 > -C(3,d)h'(\eta_1) h'(\alpha) h'(\beta) \log m - \log 2.
$$
Combining this inequality with inequality \eqref{Case3} we get
$$ m \log\Gamma + \log C_{42} < C_{44} (\log m)^4 + \log 2.$$
which yields $m < C_{45}$.

Similarly as in the cases above the assumption that $|{\it\Lambda}_3| > 0.5$ or 
$$\frac{a(n) \left(\alpha^{n-n_1} -\frac{a(n_1)}{a(n)}\right)}{b(m) \left(\beta^{m-m_1}  -\frac{b(m_1)}{b(m)}\right)}<0$$
leads in view of inequality \eqref{Case3} to 
$$
C_{40} \Gamma^{-m} \geqslant e^{\frac{1}{2}} - 1 > 0.648,
$$
which leads to   
$m  \leqslant \frac{\log \left(\frac{C_{40}}{0.648}\right)}{\log \Gamma }$.
These can be covered by the above bound $ m < C_{45}$ as long as 
$$C_{45} \geqslant \frac{\log \left(\frac{C_{40}}{0.648}\right)}{\log \Gamma}.$$

As a conclusion, if $n \geqslant N_3$ and $m \geqslant M_3$, we have $n < m < C_{45}$, where $C_{45}$ is an effectively computable constant.
Therefore, together with those finitely many cases where $n \leqslant N_4$, $m \leqslant M_4$ and all possible cases of $(m,n)$ which yield $|{\it\Phi}|, |{\it\Phi}_i|=0$ for $i=1, 2, 3$,
there can only be finitely many integers $c$ having at least two distinct representations of the form $U_n - V_m$. The number of integers $c$ and the corresponding
values of $c$ are both effectively computable. Therefore Theorem \ref{Main} is proved.

\def\cprime{$'$}


\begin{thebibliography}{10}

\bibitem{Baker:1966}
A.~Baker.
\newblock Linear forms in the logarithms of algebraic numbers. {I}, {II},
  {III}.
\newblock {\em Mathematika 13 (1966), 204-216; ibid. 14 (1967), 102-107;
  ibid.}, 14:220--228, 1967.


\bibitem{bawu93}
A.~Baker and G.~W{\"u}stholz.
\newblock Logarithmic forms and group varieties.
\newblock {\em J. Reine Angew. Math.}, 442:19--62, 1993.

\bibitem{bawu07}
A.~Baker and G.~W{\"u}stholz.
\newblock {\em Logarithmic forms and {D}iophantine geometry}, volume~9 of {\em
  New Mathematical Monographs}.
\newblock Cambridge University Press, Cambridge, 2007.

\bibitem{Bennett:2001}
M.~A. Bennett.
\newblock On some exponential equations of {S}. {S}. {P}illai.
\newblock {\em Canad. J. Math.}, 53(5):897--922, 2001.

\bibitem{Luca16}
J.~J. Bravo, F.~Luca and K.~Yaz\'{a}n.
\newblock On Pillai's problem with tribonacci numbers and powers of 2.
\newblock Preprint, 10 pages.

\bibitem{SalzburgI:2016}
K. C. Chim, I. Pink, and V. Ziegler.
\newblock On a variant of Pillai's problem.
\newblock {\em to appear in IJNT.}

\bibitem{Luca15}
M.~Ddamulira, F.~Luca, and M.~Rakotomalala.
\newblock {On a problem of Pillai with Fibonacci numbers and powers of 2}.
\newblock {\em Proceedings--Mathematical Sciences}, to appear, 9 pages.





\bibitem{Neukirch:ANTE}
J.~Neukirch.
\newblock {\em Algebraic number theory}, volume 322 of {\em Grundlehren der
  Mathematischen Wissenschaften [Fundamental Principles of Mathematical
  Sciences]}.
\newblock Springer-Verlag, Berlin, 1999.
\newblock Translated from the 1992 German original and with a note by Norbert
  Schappacher, With a foreword by G. Harder.
  
  
\bibitem{Pillai:1936}
S.~S. Pillai.
\newblock On $a^x-b^y = c$.
\newblock {\em J. Indian Math. Soc. (N.S.)}, 2:119--122, 1936.

\bibitem{Stroeker:1982}
R.~J. Stroeker and R.~Tijdeman.
\newblock Diophantine equations.
\newblock In {\em Computational methods in number theory, {P}art {II}}, volume
  155 of {\em Math. Centre Tracts}, pages 321--369. Math. Centrum, Amsterdam,
  1982.

\bibitem{wu88}
G.~W{\"u}stholz.
\newblock A new approach to {B}aker's theorem on linear forms in logarithms.
  {III}.
\newblock In {\em New advances in transcendence theory ({D}urham, 1986)}, pages
  399--410. Cambridge Univ. Press, Cambridge, 1988.

\bibitem{wu89}
G.~W{\"u}stholz.
\newblock Algebraische {P}unkte auf analytischen {U}ntergruppen algebraischer
  {G}ruppen.
\newblock {\em Ann. of Math. (2)}, 129(3):501--517, 1989.

\end{thebibliography}

 \end{document}